\documentclass[12pt]{amsart}

\usepackage{amssymb}
\usepackage{bm}
\usepackage{amsthm}
\usepackage{graphicx}
\usepackage{psfrag}
\usepackage[usenames,dvipsnames]{xcolor}

\usepackage{subcaption}
\usepackage{enumerate}
\usepackage{url}

\usepackage{algpseudocode}

\usepackage{rotating}

\usepackage{mathdots}
\usepackage{mathtools}
\usepackage{multirow}
\usepackage{blkarray}
\usepackage{relsize}
\usepackage{hyperref}

\usepackage{tikz}
\usetikzlibrary{positioning}
\usetikzlibrary{graphs,graphs.standard}
\usepackage{makecell}

\newcommand{\excise}[1]{}

\newtheorem{thm}{Theorem}[section]
\newtheorem{lemma}[thm]{Lemma}

\theoremstyle{definition}

\numberwithin{equation}{section}

\newcommand{\ring}[1]{\ensuremath{\mathbb{#1}}}

\newcommand\ZZ{\ring{Z}}

\newcommand\cA{{\mathcal A}}
\newcommand\cB{{\mathcal B}}

\newcommand\cD{{\mathcal D}}

\begin{document}

\mbox{}
\title[Atomic density of arithmetical congruence monoids]{Atomic density of arithmetical \\ congruence monoids}

\author[Olsson]{Nils Olsson}
\address{Mathematics and Statistics Department\\San Diego State University\\San Diego, CA 92182}
\email{nolsson6043@sdsu.edu}

\author[O'Neill]{Christopher O'Neill}
\address{Mathematics and Statistics Department\\San Diego State University\\San Diego, CA 92182}
\email{cdoneill@sdsu.edu}

\author[Rawling]{Derek Rawling}
\address{Mathematics and Statistics Department\\San Diego State University\\San Diego, CA 92182}
\email{drawling6679@sdsu.edu}

\date{\today}

\begin{abstract}
Consider the set $M_{a,b} = \{n \in \ZZ_{\ge 1} : n \equiv a \bmod b\} \cup \{1\}$ for $a, b \in \ZZ_{\ge 1}$.  If $a^2 \equiv a \bmod b$, then $M_{a,b}$ is closed under multiplication and known as an arithmetic congruence monoid (ACM).  A non-unit $n \in M_{a,b}$ is an atom if it cannot be expressed as a product of non-units, and the atomic density of $M_{a,b}$ is the limiting proportion of elements that are atoms.  In this paper, we characterize the atomic density of $M_{a,b}$ in terms of $a$ and $b$.  
\end{abstract}

\maketitle

\section{Introduction}
\label{sec:intro}

An arithmetic congruence monoid (ACM) is a multiplicative monoid of positive integers of the form
\[
M_{a,b} = \{n \in \ZZ_{\ge 1} : n \equiv a \bmod b\} \cup \{1\}.
\]
where $a, b \in \ZZ_{\ge 1}$ with $1 \le a \le b$ and $a^2 \equiv a \bmod b$ (the latter condition ensures $M_{a,b}$ is closed under multiplication).  A non-unit element $n \in M_{a,b}$ is an \emph{atom} if it cannot be written as a product of non-units; we denote the set of atoms by $\cA(M_{a,b})$.  We say $M_{a,b}$ is \emph{regular} if $a = 1$ (equivalently, if $\gcd(a,b) = 1$) and \emph{singular} if $a > 1$.  

Since their introduction~\cite{acmdelta,acmarithmetic,acmgenelast}, ACMs have been of interest in factorization theory, which studies the ways a given element can be written as a product of atoms~\cite{nonuniq}.  As~an example, in the \emph{Hilbert monoid} $M_{1,4}$, one can write $441 = 9 \cdot 49 = 21 \cdot 21$ as distinct factorizations since $9, 21, 49 \in \cA(M_{1,4})$.  
Regular ACMs are \emph{Krull}, which places them in an important family of atomic monoids~\cite{krullcombinatorialsurvey}, and singular ACMs are one of the first known ``simple'' examples of a monoid that can have non-accepted elasticity~\cite{acmacceptelast}.  See~the survey~\cite{acmsurvey} for an overview of the factorization properties of ACMs.  

To motivate the present manuscript, consider the following two notorious properties of the set $P \subseteq \ZZ_{\ge 0}$ of primes:\ (i) $P$ is sparse in $\ZZ_{\ge 0}$, which can be formalized as
\[
\lim_{n \to \infty} \frac{|P \cap \{1, \ldots, n\}|}{n} = 0
\]
among other ways; and (ii)~$P$ is sporadic in $\ZZ_{\ge 0}$, which is evidenced by the fact that, despite the above limit, there exists an $N \ge 2$ such that there are infinitely many pairs of primes whose difference is at most $N$ (to date, a bound of $N = 246$ is known~\cite{twinprimeprogress}, though the twin prime conjecture states that $N = 2$ is possible).  

It is natural to ask whether there are generalizations of~(i) and~(ii) for the set $\cA(M_{a,b})$ of atoms in a given ACM $M_{a,b}$.  The latter question was answered in~\cite{acmperiod}, where it was shown that the set $\cA(M_{a,b})$ forms an eventually periodic sequence if and only if $a > 1$ and $a \mid b$.  In this paper, we answer question~(i) by identifying the density of $\mathcal A(M_{a,b})$ in $M_{a,b}$.  
To this end, define the \emph{atomic density} of $M_{a,b}$ as 
\[
\cD(M_{a,b})
=
\lim_{n \to \infty}
\frac{|\cA(M_{a,b}) \cap \{1, \ldots, n\}|}{|M_{a,b} \cap \{1, \ldots, n\}|},
\]
a variant of which was defined in~\cite{algebrasatomicdensity}.  
Our main result is the following formula.  

\begin{thm}\label{t:mainthm}
If $M_{a,b}$ is an ACM and $q = \gcd(a,b)$, then 
\[
\cD(M_{a,b}) = 1-\tfrac{1}{q}.
\]
\end{thm}

A proof of Theorem~\ref{t:mainthm} is given in Section~\ref{sec:mainthm}, after some examples and an important lemma from analytic number theory in Section~\ref{sec:setup}.

\section{Setup}
\label{sec:setup}

To set the stage for the proof of Theorem~\ref{t:mainthm}, we briefly turn out attention to regular ACMs.  An element $n \in M_{1,5}$ is an atom if and only if replacing each prime in the integer factorization of $n$ with its equivalence class modulo 5 yields
\[
[1]_5,
\quad
([2]_5)^4,
\quad
([3]_5)^4,
\quad
([4]_5)^2,
\quad
([2]_5)^2[4]_5,
\quad
([3]_5)^2[4]_5,
\quad \text{or} \quad
[2]_5[3]_5.
\]
This means there is a bound (in this case, 4) on the number of primes in the integer factorization of any element of $\cA(M_{1,5})$.  The existence of such a bound implies $\mathcal A(M_{1,5})$ has density 0 by a well-known fact from analytic number theory, and thus $\cD(M_{1,5}) = 0$.  

This is an instance of a more general construction that distinguishes regular ACMs from those that are singular.  The centered expressions above are elements of the \emph{block monoid} over $G = \ZZ_5^*$ (which is the set $\cB(G)$ of commutative words with alphabet $G$ that evaluate to the identity of~$G$) that are \emph{atoms} (i.e., words that cannot be obtained by concatenating two nonempty words in $\cB(G)$).  
In fact, for any $b \ge 1$, the atoms of $M_{1,b}$ can be partitioned based on which atom of $\cB(\ZZ_b^*)$ is obtained by replacing each prime in their factorization with its equivalence class in $\ZZ_b^*$.  
This is enough to determine $\cD(M_{1,b}) = 0$, since the number of primes in the factorization of any integer in $\cA(M_{1,b})$ is bounded by the number of letters in the longest atom of $\cB(\ZZ_b^*)$ (known as the \emph{Davenport constant} of $\cB(\ZZ_b^*)$).  See~\cite[Chapter~5]{nonuniq} for a thorough overview.  

The construction described in the predecing paragraph can be made rigorous using what is known as a \emph{transfer homomorphism}; we direct the interested reader to~\cite{acmsurvey} for details.  However, atoms in a singluar ACM can have arbitrarily many primes in their integer factorization; for example, in the \emph{Meyerson monoid} $M_{4,6}$, we have 
\[
4 \cdot 7^k \in \cA(M_{4,6})
\qquad \text{for any} \qquad
k \ge 1.
\]
The key to proving Theorem~\ref{t:mainthm} turns out to be the following lemma, which implies the set of integers with a bounded number of primes in a given equivalence class $[a]_b$ has density 0 in $\ZZ_{\ge 1}$, so long as $\gcd(a,b) = 1$.  

\begin{lemma}\label{l:keydensitylemma}
Fix positive integers $a$, $b$, and $N$ with $gcd(a,b) = 1$, and let
\[
A(a,b,N) = \{n \in \ZZ : \text{at most $N$ primes in the factorization of $n$ lie in } [a]_b\}.
\]
The set $A(a,b,N)$ has density 0 in $\ZZ_{\ge 1}$, that is, 
\[
\lim_{n \to \infty}  \frac{| A(a,b,N) \cap \{1, \ldots, n\} |}{n} = 0.
\]
\end{lemma}

\begin{proof}
This follows from \cite[Theorem~08]{divisorsbook} and the fact that the sums of the reciprocals of the primes in $[a]_b$ tends to infinity.  
\end{proof}

We illustrate our intended use of Lemma~\ref{l:keydensitylemma} with an example.  
Consider the partition of the Meyerson monoid $M_{4,6}$ into 
\[
Q = \{n \in M_{4,6} : 4 \mid n\}
\qquad \text{and} \qquad
R = \{n \in M_{4,6} : 4 \nmid n\}.
\]
If $p, p' \in [5]_6$ are prime, then $4pp' = (2p)(2p')$ and is therefore reducible in $M_{4,6}$.  As~such, any atom in $Q \cap \mathcal A(M_{4,6})$ has at most one prime in $[5]_6$, meaning
\[
Q \cap \mathcal A(M_{4,6}) \subseteq A(5,6,2).
\]
This implies $Q \cap \mathcal A(M_{4,6})$ has density 0 by Lemma~\ref{l:keydensitylemma}.  
On the other hand, $R$ contains no reducible elements of $M_{4,6}$, and thus $R \subseteq \mathcal A(M_{4,6})$.  
As such, $\mathcal D(M_{4,6})$ is simply the density of $R$ in $M_{4,6}$. Since 
\[
M_{4,6} = \{2k : k \in [2]_6 \cup [5]_6\}
\qquad \text{and} \qquad
R = \{2k : k \in [5]_6\},
\]
we obtain 
\[
\mathcal D(M_{4,6})
= \lim_{n \to \infty} \frac{|R \cap \{1, \ldots, n\}|}{|M_{4,6} \cap \{1, \ldots, n\}|}
= \lim_{n \to \infty} \frac{|[5]_6 \cap \{1, \ldots, n\}|}{|([2]_6 \cup [5]_6) \cap \{1, \ldots, n\}|}
= \frac{1}{2}.
\]

\section{Proof of Theorem~\ref{t:mainthm}}
\label{sec:mainthm}

We now prove Theorem~\ref{t:mainthm}, after a short lemma.  

\begin{lemma}\label{l:thelemma}
Suppose $M_{a,b}$ is an ACM, let $q = \gcd(a,b)$, and write $a = qa'$ and $b = qb'$ with $a', b' \in \ZZ$.  The following hold:
\begin{enumerate}[(a)]
\item 
any reducible element $n \in M_{a,b}$ satisfies $q^2 \mid n$; and

\item 
$\gcd(q,b') = 1$.  

\end{enumerate}
\end{lemma}

\begin{proof}
If $n = a + kb, m = a + \ell b \in M_{a,b}$ with $k, \ell \in \ZZ_{\ge 0}$, then 
$$nm = (a + kb)(a + \ell b) = q^2(a' + kb')(a' + \ell b')$$
which proves part~(a).  For part~(b), any prime $p$ with $p \mid q$ must have $p \mid a$ and $p \nmid (a - 1)$, so since 
$$b \mid (a^2 - a) = a(a - 1),$$
any power of $p$ dividing $b$ must also divide $a$.  Since $a'$ and $b'$ cannot both have $p$ as a divisor, this forces $\gcd(q,b') = 1$.  
\end{proof}

\begin{proof}[Proof of Theorem~\ref{t:mainthm}]
If $a = b$, then $M_{a,b} = b\ZZ_{\ge 1} \cup \{1\}$, and $\cA(M_{a,b}) = b\ZZ_{\ge 1} \setminus b^2\ZZ_{\ge 1}$,~so 
\[
\cD(M_{a,b})
= \lim_{n\to\infty}\frac{|(b\ZZ_{\ge 1} \setminus b^2\ZZ_{\ge 1}) \cap \{1\ldots n\}|}{|b\ZZ_{\ge 1} \cap \{1\ldots n\}|}
= 1 - \lim_{n\to\infty}\frac{|b^2\ZZ_{\ge 1} \cap \{1\ldots n\}|}{|b\ZZ_{\ge 1} \cap \{1\ldots n\}|}
= 1 - \frac{1}{b}.
\]
Otherwise, assume $a < b$, write $a = qa'$ and $b = qb'$ for $a', b' \in \ZZ_{\ge 0}$, and let 
\[
R = \{n \in M_{a,b} : q^2 \nmid n\}.
\]
By Lemma~\ref{l:thelemma}(a) $q^2$ divides any reducible element of $M_{a,b}$, so $R \subseteq \cA(M_{a,b})$.  

We claim the set $\cA(M_{a,b}) \setminus R$ has density 0.  To this end, fix primes $p_1, p_2 \equiv a' \bmod b'$ with $p_1, p_2 \nmid q$.  If $n \in M_{a,b} \setminus R$ satisfies  $p_1p_2 \mid n$, then we can write
\[
n = q^2 p_1 p_2 m = (qp_1)(qp_2m)
\]
for some $m \in \ZZ_{\ge 1}$ with $am \equiv a \bmod b$.  This means $n$ is reducible in $M_{a,b}$ since 
\[
qp_1 = q(a' + k_1b') = a + k_1b
\qquad \text{and} \qquad
qp_2m = qm(a' + k_2b') = (am + k_2mb)
\]
for some $k_1, k_2 \in \ZZ$.  
As such, for any $n \in \cA(M_{a,b}) \setminus R$, the number of prime factors of $n$ in $[a']_{b'}$ can exceed the number of prime factors of $q$ in $[a']_{b'}$ by at most one.  Lemma~\ref{l:keydensitylemma} thus implies $\cA(M_{a,b}) \setminus R$ has density 0.  

Having shown this, we observe that
\[
M_{a,b} = \{q(a' + k b'): k \in \ZZ_{\ge 0} \} = q \cdot [a']_{b'}
\]
and conclude
\begin{align*}
\cD(M_{a,b})
&= \lim_{n\to\infty}\frac{|R \cap \{1\ldots n\}|}{|M_{a,b} \cap \{1\ldots n\}|}
= 1 - \lim_{n\to\infty}\frac{|(M_{a,b} \setminus R) \cap \{1\ldots n\}|}{|M_{a,b} \cap \{1\ldots n\}|}
\\
&= 1 - \lim_{n\to\infty}\frac{|q^2 \ZZ \cap q \cdot [a']_{b'} \cap \{1\ldots n\}|}{|q \cdot [a']_{b'} \cap \{1\ldots n\}|}
\\
&= 1 - \lim_{n\to\infty}\frac{|q \ZZ \cap [a']_{b'} \cap \{1\ldots n\}|}{|[a']_{b'} \cap \{1\ldots n\}|}
= 1 - \frac{1}{q},
\end{align*}
where the final equality follows from the fact that
$q\ZZ \cap [a']_{b'}$ contains one out of every $q$ elements of $[a']_{b'}$ since $\gcd(q,b') = 1$ by Lemma~\ref{l:thelemma}(b).  
\end{proof}

\section*{Acknowledgements}

The authors would like to thank Vadim Ponomarenko for several helpful conversations and Paul Pollock for directing us to an appropriate citation for Lemma~\ref{l:keydensitylemma}.  We~would also like to thank Adam Hoyt, Yuze Luan, and Melanie Zhang for their early computational work on the project.


\end{document}